\documentclass[12pt]{amsart}
\usepackage{etex}
\usepackage{amsfonts, amssymb, latexsym}

\usepackage{amscd,amssymb}
\usepackage{verbatim}
\usepackage{pstricks}
\usepackage{pst-node}
\usepackage[mathscr]{euscript}
\usepackage{tikz}
\usetikzlibrary{matrix,arrows}
\usepackage[normalem]{ulem}
\usepackage{varwidth}
\usepackage{leftidx}

\newsymbol\pp 1275
\usepackage{tikz}
\usetikzlibrary{matrix,arrows,backgrounds,shapes.misc,shapes.geometric,patterns,calc,positioning}

\usepackage[colorlinks=true, pdfstartview=FitV, linkcolor=blue, citecolor=blue, urlcolor=blue]{hyperref}

\usepackage{fullpage}
\usepackage{graphicx}
\usepackage{enumerate}

\def\VR{\kern-\arraycolsep\strut\vrule &\kern-\arraycolsep}
\def\vr{\kern-\arraycolsep & \kern-\arraycolsep}

\newtheorem{theorem}{Theorem}

\newtheorem{lemma}[theorem]{Lemma}
\newtheorem{prop}[theorem]{Proposition}
\newtheorem{corollary}[theorem]{Corollary}
\newtheorem{conjecture}[theorem]{Conjecture}

\theoremstyle{definition}
\newtheorem{definition}[theorem]{Definition}

\newtheorem{rmk}[theorem]{Remark}
\newenvironment{remark}[1][]{\begin{rmk}[#1]\pushQED{\qed}}{\popQED \end{rmk}}

\newtheorem{qu}[theorem]{Question}

\newtheorem*{rmknonum}{Remark}

\newtheorem{obs}[theorem]{Observation}

\newtheorem{ex}[theorem]{Example}
\newenvironment{example}[1][]{\begin{ex}[#1]\pushQED{\qed}}{\popQED \end{ex}}

\newcommand{\Hom}{\operatorname{Hom}}

\newcommand{\rep}{\operatorname{rep}}

\newcommand{\ZZ}{\mathbb Z}

\newcommand{\RR}{\mathbb R}
\newcommand{\NN}{\mathbb N}

\newcommand{\T}{\mathbf T}

\newcommand{\M}{M}

\newcommand{\F}{\mathcal{F}}
\newcommand{\A}{\mathcal{A}}

\newcommand{\Ima}{\operatorname{Im}}

\newcommand{\Id}{\mathbf{I}}

\newcommand{\Q}{\mathcal{Q}}

\newcommand{\s}{\mathcal{S}}

\newcommand{\Span}{\mathsf{Span}}

\newcommand{\jt}{\mathsf{Jt}}
\newcommand{\rkf}{\mathbf{rk_{\s}^{fil}}}
\newcommand{\eqiA}{\overrightarrow{\M_\s}}
\newcommand{\brk}{{\operatorname{\mathbf{rk}}}}
\newcommand{\sh}{\mathsf{sh}}
\newcommand{\vect}{\mathsf{vect}}
\newcommand{\pers}{\mathscr{P-}\text{persistence}}
\newcommand{\pos}{\mathscr{P}}
\newcommand{\J}{\mathscr{J}}
\newcommand{\Gscr}{\mathscr{G}}

\newcommand{\module}{\operatorname{mod}}

\newcommand{\rk}{\operatorname{rank}}

\newcommand\restr[2]{{
  \left.\kern-\nulldelimiterspace 
  #1 
  \vphantom{\big|} 
  \right|_{#2} 
  }}

\newcommand{\norm}[1]{\left\lVert#1\right\rVert}

\begin{document}
\title{The Jordan type of a multiparameter persistence module}
\author{Calin Chindris}
\address{University of Missouri-Columbia, Mathematics Department, Columbia, MO, USA}
\email[Calin Chindris]{chindrisc@missouri.edu}

\author{Min Hyeok Kang}
\address{University of Missouri-Columbia, Mathematics Department, Columbia, MO, USA}
\email[Min Hyeok Kang]{mkmb5@missouri.edu}

\author{Daniel Kline}
\address{College of the Ozarks, Mathematics Department, Point Lookout, MO, USA}
\email[Daniel Kline]{dkline@cofo.edu}

\date{\today}
\bibliographystyle{amsalpha}

\begin{abstract} In this paper, inspired by the theory of modules of constant Jordan type, we introduce the Jordan type of a persistence module. 

Let $\pos$ be a poset and $\s$ a sequence of $n$ finite substes of $\pos$. The Jordan type of a $\pers$ module $\M$ at $\s$, denoted by $\jt_\s(\M) \in \NN^n$, is defined as the Jordan type of a nilpotent operator $\T_{\M, \s}$, which is constructed from $\M$ and $\s$.  When $n=2$, we recover the notion of multirank previously introduced and studied in \cite{Thomas-thesis-2019}.

We first prove that the multirank invariants are complete for persistence modules over finite zigzag posets. This proves a conjecture of Thomas in the zigzag case.

The nilpotent operator $\T_{\M, \s}$ is functorial in $\M$. When $\pos=\ZZ^d$ or $\RR^d$, this functoriality allows us to define the Jordan filtered rank  invariant of $\M$ at $\s$. We demonstrate that these invariants are strictly finer than the classical rank invariants.  We next prove that for any two $\pers$ modules $M$ and $N$, the landscape and erosion distances between their Jordan filtered rank  invariants are bounded from above by the interleaving distance between $M$ and $N$.
\end{abstract}

\maketitle
\setcounter{tocdepth}{1}
\tableofcontents

\section{Introduction} 
At an algebraic level, one-parameter persistent homology studies persistence modules over totally ordered posets such as $(\RR, \leq)$. These modules arise from topological spaces equipped with filtrations indexed by a single parameter. For totally ordered posets, the indecomposable persistence modules are fully classified as interval modules. This classification makes it possible to define the barcode of a one-parameter persistence module, providing a concise and visual summary of its structure.

However, a single-filtered space is often insufficient to robustly encode noisy or complex data. For example, outliers in a point cloud can distort the persistence barcode when only a single parameter is used, causing misleading representations of the underlying topology (see \cite[Appendix A]{LoiCarrBlum2023}). Multiparameter persistent homology (MPH) addresses this issue by studying topological features as multiple parameters vary simultaneously. This naturally leads to persistence modules over general posets such as $(\RR^n, \leq)$. This generalization offers a more robust framework for analyzing complex datasets but introduces significant computational and theoretical challenges.  For instance, MPH involves persistence modules over wild posets, where classifying indecomposable modules is considered intractable. 

The primary goal of this paper is to define invariants of persistence modules of a general poset without relying on the indecomposable modules of the poset. Our methodology draws inspiration from the theory of modules of constant Jordan type for elementary abelian $p$-groups, as developed in \cite{BenPev, CarFriPev1, CarFriPev2, CarFRiSus} and adapted to the non-commutative setup of quiver representations in \cite{CarChiLin-2014}. This framework leads us to define the Jordan type of a persistence module at a sequence of $n$ finite subsets of the poset. 

We briefly recall just enough terminology to state our main results, with more detailed background provided in the next sections.  Let $K$ be an arbitrary field. Let $\pos$ be a poset,  and let $\s=(S_1, \ldots, S_n)$ be a sequence of $n$ pairwise disjoint finite subsets (slices) of $\pos$, where $n \geq 2$. Furthermore, we view $\s$ as the finite subposet $\cup_{i \in [n]} S_i$ of $\pos$. 

Let $u_\s$ be the nilpotent element of the poset algebra of $\pos$ defined as the sum of all basis elements of the form $p_{yx}$ with $x<y$ and $(x,y) \in S_i\times S_{i+1}$.  

We propose to study $\pers$ modules as modules over the representation-finite algebra $K[t]/(t^n)$ by first considering their restrictions to the poset $\s$, and then to the subalgebra $K[u_\s]$ generated by $u_\s$. At a basic level, this idea underlies the rich theory of modules of constant Jordan type. For a complementary approach, where the authors study persistence modules via order-embeddings of posets, we refer the reader to \cite{AmiBruHan-2025}.

Let $\M=\left( (\M_x)_{x \in \pos}, \left( \M_{yx} \right)_{x \leq y} \right)$ be a $\pers$ module such that each $\M_x$ is finite-dimensional. The finite-dimensional vector space $\restr{\M}{\s}:=\bigoplus_{x \in \s} M_x$, the restriction of $\M$ to the $\s$, comes equipped with the nilpotent operator $\T_{\M,\s}$ defined by $u_\s$. The pair $\left( \restr{\M}{\s},  \T_{\M,\s} \right)$ is precisely the restriction of $\restr{M}{\s}$ to $K[u_\s]$, viewed as a $K[t]/(t^n)$-module. Full details can be found in Section \ref{Jt-basics-sec}.

We define the \textbf{\emph{Jordan type of $\M$ at $\s$}}, denoted by $\jt_\s(\M) \in \NN^n$, to be the Jordan type of the nilpotent operator $\T_{\M,\s}$. Specifically, the $i^{th}$ coordinate of $\jt_\s(\M)$ is the number of Jordan blocks of size $i \times i$ that occur in the Jordan canonical form of $\T_{\M,\s}$. 

The operator $\T_{\M, \s}$ is functorial in $\M$. This functoriality allows us to define for every non-negative integer $i$, a new $\pers$ module $\M^i_\s$. For each $x \in \pos$, $\M^i_\s(x)$ is the image of the $i^{th}$ power of the operator $\T_{\M[x], \s}$, with $\M[x]$ being the $x$-shift of $\M$. For this construction, we assume that $\pos=\ZZ^d$ or $\RR^d$. Full details can be found in Section \ref{J-filt-rk-inv-sec}.

The \textbf{\emph{Jordan filtered rank  invariant of $\M$ at $\s$}}, denoted by $\rkf(\M)$, is defined to be the sequence of rank invariants 
\begin{equation}
\left( \brk^0_{\M, \s}, \ldots, \brk^n_{\M,\s} \right),
\end{equation}
where $ \brk^i_{\M, \s}$ denotes the rank invariant of $\M^i_\s$. The connection between $\jt_\s(\M)$ and $\rkf(\M)$ is explained in Lemma \ref{relation-Jt-Filtered-rk-inv}. Moreover, we demonstrate in Proposition \ref{J-rk-finer-rk} that the Jordan filtered rank invariant is strictly finer than the classical rank invariant for persistence modules over $d$-dimensional grids. 

We next address the completeness of the multirank invariants for finite zigzag posets. Let $\pos$ be a zigzag poset on the set $\{1, 2, \ldots, n\}$. The natural order of $\NN$ is denoted by $\leq_{\NN}$. For every pair $(i,j) \in [n]\times [n]$ with $i \leq_{\NN} j$, consider the two slices
\begin{equation}
S^{+}_{ij}:=\{ k \in \{ i, \ldots, j \} \mid k \text{~is a minimal elemement of~} \left( \{i, \dots, j\}, \leq_\pos \right) \} 
\end{equation}
and
\begin{equation}
S^{-}_{ij}:=\{ k \in \{ i, \ldots, j \} \mid k \text{~is a maximal elemement of~} \left( \{i, \dots, j\}, \leq_\pos \right) \},
\end{equation}
and define $\s_{ij}:=\left( S^{+}_{ij},  S^{-}_{ij} \right)$. Given a $\pers$ module $\M$, we associate to it the vector
\begin{equation}
R(\M):=\left( \rk(\T_{\M, \s_{ij}}) \right)_{1 \leq i \leq_{\NN} j \leq n} \in \NN^{{n+1 \choose 2}},
\end{equation}
where $\T_{\M, \s_{ij}}$ is understood to be $\Id_{\M_i}$ if $i=j$.  We point out that since $\s_{ij}$ consists of only two slices, the rank of $\T_{\M, \s_{ij}}$ is precisely the multirank of $\M$ from  $S^+_{ij}$ to $S^{-}_{ij}$.

In contrast to the classical rank invariant, which is not complete on zigzag posets, we show that $R(M)$ is complete.  This proves Thomas’s conjecture that multirank invariants completely determine linearizations of zigzag persistence sets. For more details,  see Conjecture \ref{MRL-Conj}.  

\begin{theorem}\label{main-thm-zig-zag} Let $\pos$ be a zigzag poset on the set $\{1, 2, \ldots, n\}$, and let $M$ and $N$ be two pointwise finite-dimensional $\pers$ modules. Then
$$
R(M)=R(N) \text{~if and only if~} M \simeq N.
$$
\end{theorem}
\noindent
We point out that when $n=3$, the completeness of the multirank invariants was first established in \cite[Proposition A.9]{KimMoore2024}. This particular case was, in turn, used to prove that the multirank invariant of a finitely generated $\ZZ^2$-module determines its bigraded Betti numbers (see \cite[Corollary A.10]{KimMoore2024}).

We next show that the Jordan filtered rank invariants are stable. Let $M$ and $N$ be two point-wise finite-dimensional $\pers$ modules. We define the \emph{erosion distance between $M$ and $N$ at $\s$} to be
\begin{equation}
d_{E}(M,N)_\s:=\max_{i \in \{0, \ldots, n-1
\}} d_E(\brk^i_\s(M), \brk^i_\s(N)),
\end{equation}
where $d_E$ denotes the erosion distance. The \emph{landscape distance between $M$ and $N$ at $\s$} is defined as
\begin{equation}
d_{L}(M,N)_\s:=\max_{i \in \{0, \ldots, n-1
\}} \norm{ \lambda_{M^i_\s}-\lambda_{N^i_\s} }_{\infty},
\end{equation}
where $\lambda_{X}$ denotes the landscape of a multiparameter persistence module $X$. To compare these distances to the interleaving distance, we establish in Theorem \ref{thm-stability-template} a general template for deriving erosion and landscape stability results. When applied to our set-up, it yields our next theorem.

\begin{theorem} \label{main-thm-Jt-stability} Assume $\pos=\ZZ^d$ or $\RR^d$, and let $\s=(S_1, \ldots, S_n)$ be a sequence of pairwise disjoint finite subsets of $\pos$, where $n \geq 2$. Let $M$ and $N$ be two pointwise finite-dimensional $\pers$ modules. Then
\begin{equation}
d_{L}(M,N)_\s \leq d_{E}(M,N)_\s \leq d_I(M, N),
\end{equation}
where $d_I$ is the interleaving distance.
\end{theorem}

\section{The Jordan type of a persistence module: Basic definitions and properties}\label{Jt-basics-sec}
Let $(\pos, \leq)$ be an arbitrary poset. We also view $\pos$ as a category whose objects are the elements of $\pos$, and for any two objects $x$ and $y$ in $\pos$, 
$$
\Hom_\pos(x,y)=
\begin{cases}
\{p_{yx}\} \qquad \text{~if~} x \leq y\\
\emptyset \qquad \qquad \text{~if~} x \nleq y
\end{cases}
$$

A $\pos$-persistence module (for short, $\pos$-module) $\M$ is a functor from $\pos$ (viewed as a category) to the category of vector spaces, i.e., it consists of vector spaces $\M_x$ for $x \in \pos$ and linear maps $\M_{yx}:\M_x \to \M_y$ for $x \leq y$, such that
$$
\M_{zy} \M_{yx}=\M_{zx},
$$
for all $z \geq y \geq x$, and $\M_{tt}=I_{\M_t}$ for all $t$.

The poset algebra of $\pos$, denoted by $\Lambda$, is the $K$-algebra with $K$-basis $\{p_{yx} \mid x \leq y\}$ and multiplication of basis elements given by 
$$
p_{zy} \cdot p_{y'x}  =
\begin{cases}
 p_{zx} \text{~if~} y=y'\\
 0 \text{~if~} y \neq y'
\end{cases}
$$
(We also write $e_t$ for the basis element $p_{tt}$ for $t \in \pos$.) The incidence algebra of $\pos$ can be infinite-dimensional in which case it is non-unital. Just as in the finite case, one can show that there is an equivalence of categories between the category of left $\Lambda$-modules and the category of $\pos$-modules (see \cite[Section 2.3]{BlaBruHan-2025}).

A $\pos$-module $\M$ is said to be pointwise finite-dimensional if $\M_x$ is finite-dimensional for all $x \in \pos$.  Throughout, we assume that all the $\pos$-modules are pointwise finite-dimensional.

Let $\s=(S_1, \ldots, S_n)$ be a sequence of pairwise disjoint finite subsets of $\pos$ where $n \geq 2$. For an element $x \in \pos$, we write $x \in \s$ to mean that $x \in S_i$ for some $i \in [n]$.  

Let us consider the unital finite-dimensional subalgebra $\Lambda_\s:=e_\s \Lambda e_\s \subseteq \Lambda$, where $e_\s:=\sum_{x \in \s} e_x \in \Lambda$ is the identity element of $\Lambda_\s$. (Note that this is the poset algebra of the finite poset $\s$.) Additionally, we are interested in the nilpotent element\footnote{While we have selected a specific nilpotent element $u_\s$, one could also consider other linear combinations of the elements $p_{yx}$ where $x,y \in \s$ and $x<y$.}
\begin{equation}
u_\s:=\sum_{i=1}^{n-1} \quad \sum_{\substack{(x,y) \in S_i \times S_{i+1} \\ x<y}}p_{yx} \in \Lambda_\s,
\end{equation}
which satisfies $u_\s^n=0$.

In this paper, we study $\pos$-modules by first restricting them to $\s$, and then to the subalgebra generated by $u_\s$. Specifically, let us consider the algebra homomorphism
\begin{equation}
\begin{aligned}
\varphi_\s: K[t]/(t^n) & \to \Lambda_\s\\
1 & \to e_\s\\ 
t+(t^n) & \to u_\s,
\end{aligned}
\end{equation}
where $\Ima(\varphi_\s)=K[u_\s] \leq \Lambda_\s$. 

Let $\M$ be a $\pos$-module, and let $\restr{\M}{\s}:=\bigoplus_{x \in \s}\M_x$ be the restriction of $\M$ to $\s$. We will study $\M$ via its pullback $\varphi_\s^{*}(\restr{\M}{\s})$ in $\module(K[t]/(t^n))$. It is important to note that truncated polynomial algebras in one variable are among the simplest representation-finite algebras. The algebra $K[t]/(t^n)$ has $n$ indecomposable modules, with $K[t]/(t^i)$ being the unique indecomposable $K[t]/(t^n)$-module of dimension $i$ for all $i \in [n]$.

In what follows, we describe the pullback $\varphi_\s^{*}(\restr{\M}{\s})$ in terms of a nilpotent operator. This perspective will be quite helpful for defining the Jordan type of $\M$ at $\s$ as a new invariant within the realm of invariants for multiparameter persistence modules. To this end, let
\begin{equation}
\restr{\M}{\s}(i):=\bigoplus_{x \in S_i} \M_x, \forall i \in [n],
\end{equation}
and note that $\restr{\M}{\s}=\bigoplus_{i \in [n]} \restr{\M}{\s}(i)$. Next, for each $i \in [n-1]$, consider the linear map
\begin{equation}
\T_{\M,\s}(i): \bigoplus_{x \in S_i} \M_x \to \bigoplus_{y \in S_{i+1}} \M_y
\end{equation}
whose $(y,x)$-entry is the structure morphism $\M_{yx}: \M_x \to \M_y$ if $x<y$, or the zero linear map $0: \M_x \to \M_y$ otherwise. 

Now, let us consider the nilpotent linear operator $\T_{\M,\s}$ on $\restr{\M}{\s}$ defined\footnote{We point out that for our choice of $u_\s$, its action on $\restr{\M}{\s}$ picks up information about structural morphisms of the form $\M_{yx}$ with $x<y$ and $(x,y) \in S_i \times S_{i+1}$ for some $i$. Therefore, in practice, we typically choose the slices $S_i$ so that (1) the elements of each subset $S_i$ are non-comparable; (2) for any $(x,y) \in S_i\times S_{i+1}$ with $i\leq n-1$, either $x<y$ or $x$ and $y$ are not comparable; (3) for any $x \in S_i$ with $i \leq n-1$, there exists $y \in S_{i+1}$ such that $x<y$. Under these assumptions, $u_\s$ is a sum of arrows of the Hasse quiver of the finite poset $\s$.} by
\begin{equation}\label{nil-op-as-block-matrix}
\begin{aligned}
\T_{\M, \s}: & \restr{\M}{\s} \to \restr{\M}{\s} \\
& v \to u_\s \cdot v
\end{aligned}
\end{equation}
We often think of $\T_{\M, \s}$ as the $n \times n$ block matrix whose $(i+1,i)$-block entry is the linear map $\T_{\M,\s}(i): \restr{\M}{\s}(i) \to \restr{\M}{\s}(i+1)$, and all other block entries are zero. 

\begin{definition}\label{jt-defn} Let $\M$ be a $\pos$-module, and let $\s=(S_1, \ldots, S_n)$ be a sequence of finite subsets (slices) of $\pos$ where $n \geq 2$.  We define the \emph{Jordan type of $\M$ at $\s$}, denoted by $\jt_\s(\M)$, to be the Jordan type of the nilpotent operator $\T_{\M, \s}$. Specifically, we write
$$
\jt_\s(\M)=(a_1, \ldots, a_n) \in \NN^n,
$$
where $a_i$ is the number of Jordan blocks of size $i \times i$ in the Jordan canonical form of $\T_{\M, \s}$.\\
\end{definition}

\begin{example} Let $\Gscr$ be the grid 
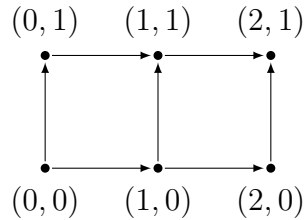
\[
\vcenter{\hbox{  
\begin{tikzpicture}[point/.style={shape=circle, fill=black, scale=.3, outer sep=3pt}, >=latex]

\node[point, label={below:$(0,0)$}] (00) at (0,0) {};
\node[point, label={below:$(1,0)$}] (10) at (1.5,0) {};
\node[point, label={below:$(2,0)$}] (20) at (3,0) {};

\node[point, label={$(0,1)$}] (01) at (0,1.5) {};
\node[point, label={$(1,1)$}] (11) at (1.5,1.5) {};
\node[point, label={$(2,1)$}] (21) at (3,1.5) {};

\path[->] 
(00) edge (10)
(10) edge (20)
(01) edge (11)
(11) edge (21)
(00) edge (01)
(10) edge (11)
(20) edge (21);

\end{tikzpicture}
}}
\]
and $\s=(S_1, S_2, S_3)$ where $S_1=\{(0,1), (1,0)\}$, $S_2=\{(1,1), (2,0)\}$, and $S_3=\{(2,1)\}$.

For
\[
M:=~
\vcenter{\hbox{  
\begin{tikzpicture}[
    node style/.style={align=center},
    arrow/.style={->}
]

\node[node style] (A1) at (0,1.75) {$K$};
\node[node style] (B1) at (1.5,1.75) {$K^2$};
\node[node style] (C1) at (3,1.75) {$K$};

\node[node style] (A2) at (0,0) {$0$};
\node[node style] (B2) at (1.5,0) {$K$};
\node[node style] (C2) at (3,0) {$K$};

\draw[arrow] (A1) -- (B1) node[midway, above] {$\begin{pmatrix} 0 \\ 1 \end{pmatrix}$};
\draw[arrow] (B1) -- (C1) node[midway, above] {$\begin{pmatrix} 1 & 1 \end{pmatrix}$};

\draw[arrow] (A2) -- (B2);
\draw[arrow] (B2) -- (C2) node[midway, below] {$1$};

\draw[arrow] (A2) -- (A1);
\draw[arrow] (B2) -- (B1) node[midway, left] {$\begin{pmatrix} 1 \\ 0 \end{pmatrix}$};
\draw[arrow] (C2) -- (C1) node[midway, right] {$1$};
\end{tikzpicture}
}}
\text{~we have that~} \jt_\s(M)=(1,1,1).
\]
\end{example}

The block entries on the subdiagonal of $\T_{\M, \s}$ can be arranged to form a representation of the equoriented $\mathbb{A}_n$ quiver:
\begin{equation}
 \eqiA:\qquad \restr{\M}{\s}(1) \xrightarrow{\T_{\M,\s}(1)} \restr{\M}{\s}(2) \to \ldots  \to \restr{\M}{\s}(n-1) \xrightarrow{\T_{\M,\s}(n-1)} \restr{\M}{\s}(n)
\end{equation}

The following result relates the Jordan type of $\M$ to the decompositions of $\varphi_\s^*(\M)$ and $\eqiA$ into indecomposable representations. 

\begin{prop}\label{prop-basic-jt} Let $\pos$ be any poset, $\s=(S_1, \ldots, S_n)$ an $n$-sequence of finite slices, and $\M$ a $\pos$-module. The following statements are equivalent:
\begin{enumerate}[(a)]
\item $\jt_\s(\M)=(a_1, \ldots, a_n) \in \NN^n$;
\smallskip
\item the pullback $\varphi^*_\s(\restr{\M}{\s})$ decomposes into indecomposable $K[t]/(t^n)$-modules as 
\begin{equation}
\varphi^*_\s(\restr{\M}{\s}) \simeq \bigoplus_{i=1}^n \left( K[t]/(t^i) \right)^{a_i};
\end{equation}
\smallskip

\item the number of indecomposable representations of dimension $i$ in the decomposition of the representation $\eqiA$ is $a_i$ for all $i \in [n]$.
\end{enumerate}
\end{prop}

\begin{proof} $\boxed{(a) \Longleftrightarrow (b)}$ Any finite-dimensional $K[t]/(t^n)$-module can be viewed as a pair $(V,T)$ with $V$ a finite-dimensional vector space and $T:V \to V$ a nilpotent linear operator with $T^n=0$. Under this correspondence, the indecomposable direct summands correspond to the Jordan blocks in the Jordan canonical decomposition of $T$. Furthermore, the indecomposable module $K[t]/(t^i)$ corresponds to $(K^i, J_i(0))$, where $J_i(0)$ denotes the $i \times i$ Jordan block with eigenvalue $0$.The equivalence of $(a)$ and $(b)$ now follows.

\bigskip
\noindent
$\boxed{(b) \Longleftrightarrow (c)}$ Let consider the algebra homomorphism
\begin{equation}
\begin{aligned}
\psi: K[t]/(t^n) & \to K \Q\\ 
t+(t^n) & \to \alpha_1+\ldots+\alpha_{n-1},
\end{aligned}
\end{equation}
where $K\Q$ is the path algebra of $\Q=\overrightarrow{\mathbb{A}}_n$, and $\alpha_1, \ldots, \alpha_{n-1}$ are the arrows of $\Q$.  It is easy to see that the pullback along $\psi$ of any $i$-dimensional indecomposable representation of $\Q$ is isomorphic to $K[t]/(t^i)$. Combining this observation with the fact that $\psi^*(\overrightarrow{\restr{\M}{\s}})=\varphi^*_\s(\restr{\M}{\s})$, we get that $(b)$ is indeed equivalent to $(c)$.
\end{proof}

\begin{rmk} For the $\mathbb{A}_n$ quiver with all arrows pointing to the right, take $S_i$ to be the singleton consisting of the $i^{th}$ vertex for each $i \in [n]$. If $\jt_\s(\M)=(a_1, \ldots, a_n)$ then $a_i$ is number of points in the persistence diagram of $\M$ that lie on the $i^{th}$ diagonal $y=x+i-1$.
\end{rmk}

As an immediate consequence, we obtain the following additivity property of the Jordan type.

\begin{corollary}\label{Jt-additivity-cor}
\begin{enumerate}[(a)]
\item Let $\M_1$ and $\M_2$ be two $\pos$-modules. Then
\begin{equation}\label{jt-add}
\jt_\s(\M_1 \oplus \M_2)=\jt_\s(\M_1)+\jt_\s(\M_2).
\end{equation}

\item Let $0 \to M \to N \to L \to 0$ be a short exact sequence of $\pos$-modules. Then 
$$
\jt_\s(N) =\jt_\s(M)+\jt_\s(L)
$$
if and only if the pullback 
$$
0 \to \varphi_\s^*(\restr{M}{\s}) \to\varphi_\s^*(\restr{N}{\s}) \to \varphi_\s^*(\restr{L}{\s}) \to 0
$$
splits in $\module(K[t]/(t^n))$.
\end{enumerate}
\end{corollary}

\begin{proof} $(a)$ We have that
$$
\varphi_\s^*(\restr{(M_1\oplus M_2)}{\s}) \simeq \varphi_\s^*(\restr{M_1}{\s}) \oplus \varphi_\s^*(\restr{M_2}{\s}),
$$
which together with Proposition \ref{prop-basic-jt} yields $(\ref{jt-add})$.

\bigskip
\noindent
$(b)$ Using Proposition \ref{prop-basic-jt} again, we obtain that for any three $\pos$-modules $M,N$, and $L$,
\begin{equation}
\jt_\s(N)=\jt_\s(M)+\jt_\s(L) \Longleftrightarrow \varphi_\s^*(\restr{N}{\s}) \simeq \varphi_\s^*(\restr{M}{\s}) \oplus \varphi_\s^*(\restr{L}{\s}). 
\end{equation}

Next, the sequence $0 \to \varphi_\s^*(\restr{M}{\s}) \to\varphi_\s^*(\restr{N}{\s}) \to \varphi_\s^*(\restr{L}{\s}) \to 0$ is exact, and therefore, it splits if and only if 
$$\varphi_\s^*(\restr{N}{\s}) \simeq \varphi_\s^*(\restr{M}{\s}) \oplus \varphi_\s^*(\restr{L}{\s})
$$ 
by  \cite[Theorem 1]{Miyata-67}. The proof now follows.
\end{proof}

\section{The completeness of the multirank invariants  for finite zigzag posets} In this section, $\pos$ is a zigzag poset on the set $\{1, \ldots, n\}$, \emph{i.e.}, the Hasse quiver $\Q_\pos$ of $\pos$ is a Dynkin quiver of type $\mathbb{A}_n$ with arbitrary orientation. Recall that for every $(i,j)\in [n] \times [n]$ with $i \leq_{\NN} j$, we have the pair of slices
$$
\s_{ij}:=\left( S^{+}_{ij},  S^{-}_{ij} \right),
$$
where
$$
S^{+}_{ij}:=\{ k \in \{ i, \ldots, j \} \mid k \text{~is a minimal elemement of~} \left( \{i, \dots, j\}, \leq_\pos \right) \} 
$$
and
$$
S^{-}_{ij}:=\{ k \in \{ i, \ldots, j \} \mid k \text{~is a maximal elemement of~} \left( \{i, \dots, j\}, \leq_\pos \right) \}.
$$

For a $\pers$ module $\M$, we are interested in the invariant vector
\begin{equation}
R(\M):=\left( \rk(\T_{\M, \s_{ij}}) \right)_{\substack{(i,j) \in [n] \times [n] \\ i \leq_{\NN} j}}
\end{equation}

\noindent
Note that if $i \leq_\NN j$, and $i$ and $j$ are comparable in $\pos$, then $S^{+}_{ij}$ and $S^{-}_{ij}$ are singletons and 
$$
\rk(\T_{\M, \s_{ij}})=\rk(M_{ji}) \text{~or~} \rk(M_{ij}).
$$

\begin{example} Assume $n=7$, and let
$$
Q_\pos:~
\vcenter
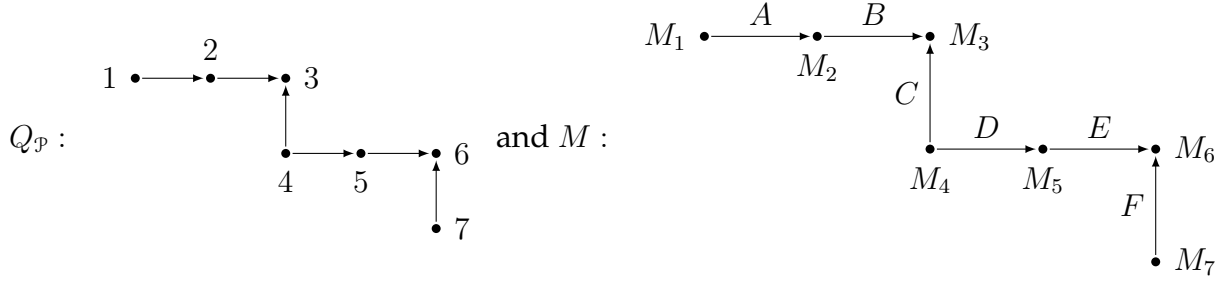
{\hbox{  
\begin{tikzpicture}[point/.style={shape=circle, fill=black, scale=.3pt,outer sep=3pt},>=latex]
   \node[point, label={left:$1$}] (1) at (0,1) {};
   \node[point, label={above:$2$}] (2) at (1,1) {};
   \node[point, label={right:$3$}] (3) at (2,1) {};
   \node[point, label={below:$4$}] (4) at (2,0) {};
   \node[point, label={below:$5$}] (5) at (3,0) {};
   \node[point, label={right:$6$}] (6) at (4,0) {};
   \node[point, label={right:$7$}] (7) at (4,-1) {};
   
   \path[->]
  (1) edge  (2)
  (2) edge (3)
  (4) edge (3)
  (4) edge  (5)
  (5) edge (6)
  (7) edge (6);
\end{tikzpicture} 
}}
\text{~and~}
\M:~
\vcenter{\hbox{  
\begin{tikzpicture}[point/.style={shape=circle, fill=black, scale=.3pt,outer sep=3pt},>=latex]
   \node[point, label={left:$\M_1$}] (1) at (0,1.5) {};
   \node[point, label={below:$\M_2$}] (2) at (1.5,1.5) {};
   \node[point, label={right:$\M_3$}] (3) at (3,1.5) {};
   \node[point, label={below:$\M_4$}] (4) at (3,0) {};
   \node[point, label={below:$\M_5$}] (5) at (4.5,0) {};
   \node[point, label={right:$\M_6$}] (6) at (6,0) {};
   \node[point, label={right:$\M_7$}] (7) at (6,-1.5) {};
   
   \path[->]
  (1) edge node[above] {$A$}  (2)
  (2) edge node[above] {$B$} (3)
  (4) edge node[left] {$C$} (3)
  (4) edge node[above] {$D$} (5)
  (5) edge node[above] {$E$} (6)
  (7) edge node[left] {$F$}(6);
\end{tikzpicture} 
}}
$$

If $i=4$ and $j=6$, the corresponding slices are $S^+_{46}=\{4\}$ and $S^{-}_{46}=\{6\}$, and the $(4,6)$-coordinate of $R(\M)$ is equal to $\rk \left( \M_{64} \right)=\rk \left( ED \right)$.

As another example, if $i=1$ and $ j=7$, then $S^+_{17}=\{1,4,7\}$ and $S^{-}_{17}=\{3,6\}$, and the $(1,7)$-coordinate of $R(\M)$ is equal to
\[
\rk \begin{pmatrix}
\begin{array}{ccc}
BA & C & 0 \\
0 & ED & F
\end{array}
\end{pmatrix}. 
\]

\end{example}

\begin{example} Assume $n=4$ and let 
$$
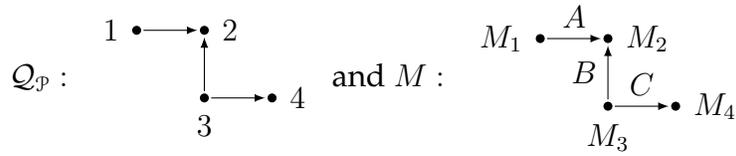
\Q_\pos:~
\vcenter{\hbox{  
\begin{tikzpicture}[point/.style={shape=circle, fill=black, scale=.3pt,outer sep=3pt},>=latex]
   \node[point, label={left:$1$}] (1) at (0,.9) {};
   \node[point, label={right:$2$}] (2) at (.9,.9) {};
   \node[point, label={below:$3$}] (3) at (.9,0) {};
   \node[point, label={right:$4$}] (4) at (1.8,0) {};
   
   \path[->]
  (1) edge  (2)
  (3) edge (2)
  (3) edge (4);
\end{tikzpicture} 
}}
\text{~and~}
\M:~
\vcenter{\hbox{  
\begin{tikzpicture}[point/.style={shape=circle, fill=black, scale=.3pt,outer sep=3pt},>=latex]
   \node[point, label={left:$\M_1$}] (1) at (0,.9) {};
   \node[point, label={right:$\M_2$}] (2) at (.9,.9) {};
   \node[point, label={below:$\M_3$}] (3) at (.9,0) {};
   \node[point, label={right:$\M_4$}] (4) at (1.8,0) {};
   
   \path[->]
  (1) edge node[above] {$A$}  (2)
  (3) edge node[left] {$B$} (2)
  (3) edge node[above] {$C$} (4);
\end{tikzpicture} 
}}
$$

Then, we may identify $R(M)$ with the following upper triangular matrix:
\[
R(M)=\begin{bmatrix}
\dim M_1 & \rk(A) &  \rk \left(\begin{matrix}
A&B
\end{matrix}
\right)  & \rk \left(\begin{matrix}
A & B \\ 
0 & C
\end{matrix}
\right)\\

\\
      & \dim M_2 & \rk(B) &  \rk \left(\begin{matrix}
B \\ C
\end{matrix}
\right) \\

\\
      &       & \dim M_3 & \rk(C) \\
\\
    &     &     & \dim M_4
\end{bmatrix} \in \NN^{10}
\]
\smallskip
\noindent
where the $i^{th}$ row records the coordinates of $R(M)$ indexed by $(i, i),\ldots (i,4)$ for each $i \in [4]$.
\end{example}

The next result plays a crucial role in the proof of Theorem \ref{main-thm-zig-zag}.

\begin{prop}\label{main-tech-result} Keep the same notation as above. Then
\begin{equation}\label{eqn-span-prop-key}
\Span_\ZZ \left\lbrace R(M) \mid M \in \rep(\Q_{\pos}) \right\rbrace=\ZZ^{\{(i,j) \mid 1 \leq i \leq_\NN j \leq n \} }.
\end{equation}
\end{prop}

\begin{proof} Let us denote $\Span_\ZZ \left\lbrace R(M) \mid M \in \rep(\Q_{\pos}) \right\rbrace$ by $\mathscr{L}$. For $1\leq i \leq_{\mathbb{N}}j \leq n$, we use $I_{[i,j]}$ to denote the interval module supported on the interval $[i,j]$, i.e. we assign a copy of the base field on each vertex between $i$ and $j$ and set our nontrivial linear maps to be the identity. Viewing $R(M)$ as an upper triangular matrix, we will prove that the standard basis matrices $\{e_{ij},\, 1 \leq i \leq_\NN j \leq n\}$ are contained in $\mathscr{L}$ by inducting on $n\geq 3$. For the case $n=3$, we note that $R(I_{[i,i]})=e_{ii}$ for $i=1,2,3$. Suppose $\Q_{\pos}$ is equioriented. One can easily check that $e_{12}$, $e_{23}$, and $e_{13}$ can be obtained by sequentially examining $R(I_{[1,2]})$, $R(I_{[2,3]})$, and $R(I_{[1,3]})$. If the arrows in $\Q_{\pos}$ differ in orientation, we make use of $R(I_{[2,3]})$ to obtain $e_{12}$, $e_{13}$, and $e_{23}$ by sequentially examining $R(I_{[1,3]})$, $R(I_{[1,2]})$, then back to $R(I_{[1,3]})$. 

Assume $n>3$. As before, $R(I_{[i,i]})=e_{ii}$ for $i \in [n]$, so we will not concern ourselves with any diagonal entries for the remainder of the proof.

Consider two cases depending on the orientation of the last two arrows of $\Q_{\pos}$. Suppose they share the same orientation. An immediate consequence is that $R(I_{[n-1,n]})=e_{n-1\,n-1} + e_{n-1\,n}+e_{nn}$, from which we find $e_{n-1\,n}$. Let $\Q_1$ be the type $\mathbb{A}_{n-1}$ quiver obtained by deleting the last arrow from $\Q_{\pos}$. Given a $\Q_1$-representation $M_1$, extending it to the $\Q_{\pos}$-representations $M$ and $M'$ via the identity map $\mathrm{id}_{M_{n-1}}$ and zero map respectively yields the following matrix extensions:
$$
R(M)=\begin{bmatrix}
    \begin{array}{c|c}
       \hspace{.4in}  R(M_1)  \hspace{.4in} &  \color{red}\begin{matrix}
        a_{{1\,n-1}} \\ a_{2\,n-1} \\ a_{3\,n-1} \\ \vdots \\ a_{n-2\,n-1} \\ a_{n-1\,n-1}
    \end{matrix}\color{black}\\
    \hline 
    \begin{matrix}
        0 & 0  & \cdots & 0
    \end{matrix}& 
    \end{array}
\end{bmatrix} \quad \text{~and~} \quad R(M')=\begin{bmatrix}
    \begin{array}{c|c}
       \hspace{.45in}  R(M_1)  \hspace{.45in} & \color{red}\begin{matrix}
        a_{1k} \\ \vdots \\ a_{k-1\,k} \\ 0 \\ \vdots \\ 0 
    \end{matrix}\color{black}\\
    \hline 
    \begin{matrix}
        0 & 0  & \cdots & 0
    \end{matrix}& 
    \end{array}
\end{bmatrix}
$$
where $R(M_1)=(a_{ij})$ and $k<n-1$ is the largest vertex at which the orientation of arrows change. Since the $n^{th}$ column of $R(M)$ is determined by the $(n-1)^{th}$ column of $R(M_1)$, we apply the induction hypothesis to deduce that $e_{ij},\,i<j,j<n-1$, and $e_{i\,n-1}+e_{in},\,1\leq i<n-1$, live in $\mathscr{L}$. Similarly, from $R(M')$, we find $e_{i\,n-1}$, and hence $e_{in}$, to be in $\mathscr{L}$ for $1\leq i <n-1$.

Now, suppose the last two arrows of $\Q_{\pos}$ have differing orientation. Note that $n-1$ is the largest vertex at which the orientation of arrows change. Therefore,
$$
R(M')=\begin{bmatrix}
    \begin{array}{c|c}
       \hspace{.4in}  R(M_1)  \hspace{.4in} & \color{red} \begin{matrix}
        a_{{1\,n-1}} \\ a_{2\,n-1} \\ a_{3\,n-1} \\ \vdots \\ a_{n-2\,n-1} \\ 0
    \end{matrix}\color{black}\\
    \hline 
   \begin{matrix}
        0 & 0  & \cdots & 0
    \end{matrix} & 
    \end{array}
\end{bmatrix}.
$$
Once again, by induction hypothesis, we find $e_{ij},\,i<j<n-1$ and $e_{i\,n-1}+e_{in},\,1\leq i<n-1$, in $\mathscr{L}$. It suffices to show that $e_{in},\,1\leq i \leq n-1$ lie in $\mathscr{L}$. 

\textbf{Case 1:} If the second and third to last arrows share the same orientation, $R(I_{[n-2,n]})-R(I_{[n-1,n]})=e_{n-2\,n-2}+e_{n-2\,n-1}$, yielding $e_{n-2\,n-1}$, and thus, $e_{n-2\,n}$. Consider the type $\mathbb{A}_{n-1}$ quiver $\Q_2$ obtained by deleting the second to last arrow from $\Q_{\pos}$ and joining the $n-2$ and $n-1$ vertices. Given a $\Q_2$-representation $M_2$, extend it to the $\Q_{\pos}$-representation $M$ via the identity map $\mathrm{id}_{M_{n-2}}$. We observe that $R(M)$ resembles $R(M_2)$ with an additional column essentially identical to the $(n-2)^{th}$ column of $R(M_2)$:
$$
R(M)=\begin{bmatrix}
    \begin{array}{c|c|c|c}
        \begin{matrix}
        \\ \hspace{.3in} \ddots \hspace{.3in}\\ \\  \\ \\ \\  \\ 
        \end{matrix} & \begin{matrix} a_{1\,n-2} \\ a_{2\,n-2} \\ \vdots  \\ a_{n-3\,n-2} \\ \bullet \\ \\ \\
        \end{matrix} & \color{red}\begin{matrix}  a_{1\,n-2} \\ a_{2\,n-2} \\ \vdots  \\ a_{n-3\,n-2} \\ \bullet  \\ \bullet \\ \\ 
        \end{matrix} \color{black} & \begin{matrix}
            a_{1\,n-1} \\ a_{2\,n-1} \\ \vdots \\ a_{n-3\,n-1} \\ \color{red} \bullet \color{black} \\ a_{n-2\,n-1} \\ \bullet\\ 
        \end{matrix}
    \end{array}
\end{bmatrix}.
$$ Since we already have $e_{n-2\,n-1},e_{n-2\,n} \in \mathscr{L}$, our induction hypothesis gives us that $e_{in} \in \mathscr{L}$ for all $1 \leq i\leq n-1$.

\textbf{Case 2:} Suppose the last three arrows alternate in orientation. We will show that all $e_{in},\,1 \leq i \leq n-1$, live in $\mathscr{L}$ by inducting backwards on $i$. The case $i=n-1$ follows from $R(I_{[n-2,n]})= e_{n-1\,n} + \sum_{j=n-2}^n e_{jj} + \sum_{\ell=1}^{n-2} (e_{\ell\,n-1}+e_{\ell n})$. Assume $i<n-1$ and that $e_{jn}$, and thus $e_{j\,n-1}$, lie in $\mathscr{L}$ for all $j>i$. Consider the representation $I_{[1,n]}$. Let $1=i_0<i_1<i_2<...<i_{m-1}<i_m=n$ be all vertices at which the orientation of arrows change. The arrows between any pair $i_\ell, i_{\ell+1}$ share the same orientation, so $R(I_{[i_{\ell},i_{\ell+1}]})$ contains a $(i_{\ell+1}-i_{\ell}+1)\times (i_{\ell+1}-i_{\ell}+1)$ upper triangular matrix of $1$'s given by compositions of the identity map $\mathrm{id}_K$. With this knowledge, we can represent $R(I_{[1,n]})$ as the block matrix
$$
\begin{bmatrix}
\begin{array}{ccc|cccc|c|c}
\ddots & & & \vdots & & \vdots& \vdots & \vdots & \vdots\\
&\bullet & 1 & 1 & \cdots & 1 & 1 & 2 & 2\\ \hline
&& \color{red}\bullet & 1 & \cdots &  & 1 & 1& 2\\
&&&\ddots & \ddots & & \vdots & \vdots & \vdots \\
&&&&\bullet & 1 & 1 &1 & 2\\
&&&&&\bullet & 1 &1 & 2 \\ \hline
&&&&&&\color{red}\bullet & 1 & 1\\ \hline
&&&&&&&\color{red}\bullet & 1\\ \hline
&&&&&&&& \color{red} \bullet
\end{array}
\end{bmatrix}
$$
where, ignoring the diagonal entries, the repeated entries within each block correspond to the rank of the same map. The colored bullets correspond to the $(i_{\ell},i_{\ell})$-th entries.

Let $R(I_{[1,n]})=(a_{jk})$. For any $j<n-1$, we claim that either $a_{jn}=a_{j\,n-1}$ or $a_{jn}=a_{j\,n-1}+1$. Since all linear maps in $I_{[1,n]}$ are the identity and the arrows between each pair $i_{\ell},i_{\ell+1}$ share the same orientation, by replacing these same-oriented linear maps with a single identity map, we may reduce $I_{[1,n]}$ to an interval module $I$ on a type $\mathbb{A}_{m+1}$ quiver whose arrows \pagebreak 

\noindent are all alternating. Then, $R(I)$ is given by the matrix
$$\begin{bmatrix}
\ddots & & & \ddots & & & \ddots \\
&\bullet & 1 & 1 & 2 & 2 & 3 & 3 & 4\\
&&\bullet & 1 & 1 & 2 & 2 & 3 & 3\\
&&&\bullet & 1 & 1 & 2 & 2 & 3\\
&&&&\bullet & 1 & 1 & 2 & 2\\
&&&&&\bullet & 1 & 1 & 2\\
&&&&&&\bullet & 1 & 1\\
&&&&&&&\bullet & 1\\
&&&&&&&&\bullet
\end{bmatrix}.
$$ 
This matrix can also be obtained from the prior block matrix by keeping the colored diagonal entries then compressing each block to a single entry. Our claim now follows by observing the last two columns of $R(I)$. 

Since $i<n-1$, we must have $a_{in}=a_{i\,n-1}$ or $a_{in}= a_{i\,n-1}+1$. If $i_{\ell}<i<i_{\ell+1}$ for some $\ell$, consider $R(I_{[i,n]})=(b_{jk})$. Note that $b_{j\,n-1}=a_{i_{\ell+1}\,n-1}$ and $b_{jn}=a_{i_{\ell+1}n}$ for all $j<i$. In particular, if $a_{i\,n} = a_{i\,n-1}$, $b_{jn} = b_{j\,n-1}+1$ for all $j<i$. Otherwise, $b_{jn}=b_{j\,n-1}$ for all $j<i$. Let us first examine the former case. Recall that we have $e_{j\,n-1}+e_{jn} \in \mathscr{L}$ for all $1\leq j<n-1$. Accordingly, reduce $R(I_{[i,n]})$ to $\sum_{j=1}^{i-1} e_{jn}$, and from the calculation $R(I_{[n-1,n]})- \sum_{j=1}^{i-1} e_{jn} = e_{n-1\,n-1}+e_{nn} + \sum_{j=i}^n e_{jn}$, we find $e_{in} \in \mathscr{L}$. In the latter case, we can directly reduce $R(I_{[i,n]})$ to obtain $e_{in}$. 

Now, assume $i=i_{\ell}$ for some $\ell$. If $\ell=0$, $i=1$, so we are done by applying our induction hypothesis to $R(I_{[n-1,n]})$. Suppose otherwise, and consider $R(I_{[i_{\ell-1},n]})=(c_{jk})$. Note that $c_{j\,n-1} = a_{i_{\ell-1}\,n-1}$ and $c_{jn}=a_{i_{\ell-1},n}$ for all $j<i$. In the case $a_{in}=a_{i\,n-1}$, $c_{jn}=c_{j\,n-1}+1$  for all $j<i$, and otherwise, $c_{jn}=c_{j\,n-1}$ for all $j<i$. We conclude the proof by repeating the prior arguments made for $(b_{jk})$.
\end{proof}

We are now ready to prove our first theorem.

\begin{proof}[\textbf{Proof of Theorem \ref{main-thm-zig-zag}}]
Let $\Q$ be the Hasse quiver of $\pos$, which is of type $\mathbb{A}_n$.
Since the quantities $\rk(\T^i_{\M, \s})$ are additive on direct sums, we get a homomorphism
$$
\begin{aligned}
R: K^{sp}(\Q)& \to \ZZ^{{n+1 \choose 2}}\\
[\M] & \to R(\M),
\end{aligned}
$$
where $K^{sp}(\Q)$ is the spilt Grothendieck's group of the category of finite-dimensional representations of $\Q$. This is a free abelian group with a $\ZZ$-basis consisting of the classes of the indecomposable representations. 

From Gabriel's Theorem, we know that 
$$
\brk_\ZZ K^{sp}(\Q)={n+1 \choose 2},
$$ 
and Proposition \ref{main-tech-result} tells us that $R$ is surjective. Therefore, the kernel of $R$ is of rank zero, and thus $R$ is an isomorphism. 
\end{proof}

\begin{remark}[\textbf{Linearizations of persistence sets}] Let $\pos$ be any finite poset.  A $\pos$-persistence set $B$ is a functor from $\pos$ to the category of finite sets, \emph{i.e.}, we have a finite set $B_x$ for every $x \in \pos$ and a function $f_{yx}: B_x \to B_y$ for every pair $(x,y)$ with $x \leq y$ such that $f_{xx}=\Id_{B_x}$ for all $x \in \pos$, and $f_{zy} \circ f_{yx} = f_{zx}$ for all $z \geq y \geq x$.

Let $B$ be a $\pos$-persistence set, and let $L_B$ be its linearization,\emph{ i.e.}, $L_B$ is the $\pos$-persistence module defined by 
$$
(L_B)_{z} \text{ is the vector space with basis } B_z, \forall z \in \pos,
$$
and the linear maps are induced by the functions $f_{yx}:B_x \to B_y$ at the level of bases.  Such linearizations arise naturally in the context of clustering of points in finite metric spaces (see \cite[Section 2.2]{Alonso-Kerber-pers-sets-2023} for full details).

In what follows, if $\s=(S_1,S_2)$ is a pair of slices and $\M$ is a $\pos$-module, we write 
$$
\mathrm{multirank}_{\M}(S_1,S_2) \text{~for the rank of~}\T_{\M, \s}.
$$ 
This is the notation for multiranks, introduced and studied in \cite{Thomas-thesis-2019}.

\begin{conjecture}[\textbf{The Multirank Conjecture for Linearizations}](see \cite[3.2.12]{Thomas-thesis-2019}) \label{MRL-Conj}
Let $\pos$ be a finite poset. If $B_1$ and $B_2$ are two $\pos$-persistence sets such that 
$$
\mathrm{multirank}_{L_{B_1}}(S_1,S_2) = \mathrm{multirank}_{L_{B_2}}(S_1,S_2), 
$$
for all pairs of slices $(S_1, S_2)$, then 
$$
L_{B_1} \cong L_{B_2}.
$$
\end{conjecture}

Our Theorem \ref{main-thm-zig-zag} proves this conjecture when $\pos$ is a zigzag poset. 
\end{remark}

\section{The Jordan filtered rank invariant}\label{J-filt-rk-inv-sec}
Throughout this section, we assume that $\pos=\ZZ^d$ or $\RR^d$. Let $\s=(S_1, \ldots, S_n)$ be a tuple of $n$ pairwise disjoint slices of $\pos$. We begin with the following simple but very useful lemma.

\begin{lemma} \label{main-lemma-functoriality} Let $f: M \to N$ be a homomorphism of $\pos$-modules. Then
\begin{equation}\label{funct-eqn}
\restr{f}{\s}\left( \Ima \left( \T^i_{M,\s} \right) \right) \subseteq \Ima\left(\T^i_{N,\s} \right),
\end{equation}
for all integers $i \in \ZZ_{\geq 0}$. 
\end{lemma}
\begin{proof} By construction, for any $\pos$-module $L$, we have that
$$
\Ima \left( \T^i_{L,\s} \right)=\{u^i_\s \cdot v \mid v \in \restr{L}{\s} \}.
$$
As a homomorphism of $\pos$-modules, $f$ is $\Lambda_\s$-linear, and thus $f(u^i_\s \cdot v)=u^i_\s \cdot f(v) \in \Ima\left(\T^i_{N,\s} \right)$ for any $v \in \restr{M}{\s}$. The proof of the claim now follows.
\end{proof}

Now, let $\M_\s$ be the $\pos$-module defined by
\begin{equation}
\left( \M_\s \right)_x:=\restr{\M[x]}{\s}=\bigoplus_{z \in \s}\M_{x+z} \text{~and~}\left( \M_\s \right)_{y,x}:=\bigoplus_{z \in \s} \M_{y+z, x+z}, \forall x \leq y.
\end{equation}

For each $i \in \NN$, equation $(\ref{funct-eqn})$ yields a functor $\M \to \Ima \left( \T^i_{M,\s} \right)$, from the category of pointwise finite-dimensional $\pos$-modules to $\vect_K$. We are going to use this functoriality property to construct a weakly decreasing filtration of $\M_\s$.

\begin{definition}[\textbf{Jordan modules}]\label{defn-ith-module-at-s}
For a $\pos$-module $\M$ and $i \in \NN$, we define a new $\pos$-module $\M^i_\s$ by
\begin{equation}
\left( \M^i_{\s} \right)_x:=\Ima(\T^i_{\M[x], \s}) \leq \restr{M[x]}{\s},  \forall x \in \pos,
\end{equation}
and
\begin{equation}
\left( \M^i_{\s} \right)_{y,x}:=\restr{\sh^{y-x}_{\M[x]}}{\Ima \left( \T^i_{\M[x],\s} \right)}:  \Ima \left( \T^i_{\M[x],\s} \right) \to \Ima\left(\T^i_{\M[y],\s} \right), \forall x \leq y.
\end{equation}
\end{definition}
We also refer to $\M^i_\s$ as the \textbf{degree $i$ Jordan module} associated to $\M$ and $\s$.

\bigbreak
\noindent
It is immediate to see that
\begin{equation}\label{filt-Jordan-modules}
0=\M^n_{\s} \subseteq \M^{n-1}_{\s} \subseteq \ldots \subseteq \M^1_{\s} \subseteq \M^0_{\s}=\M_\s
\end{equation}

We are now ready to define the Jordan filtered rank invariant of $\M$ at $\s$.

\begin{definition}[\textbf{The Jordan filtered rank  invariant at $\s$}] Let $\M$ be a $\pos$-module and $i \in \NN$. We define the \emph{degree $i$ rank invariant of $\M$ at $\s$} to be the rank invariant of $\M^i_{\s}$
\begin{equation}
\begin{aligned}
\brk^i_{\M, \s}:& \pos\times \pos \to \NN \cup \{\infty\}\\
& (x,y) \to
\begin{cases}
\rk \left((\M^i_\s)_{y,x} \right) &\text{~if~} x \leq y\\
\infty&\text{~otherwise~}
\end{cases}
\end{aligned}
\end{equation} 

The \emph{Jordan filtered ($\J$-filtered) rank invariant of $\M$ at $\s$} is defined as
\begin{equation}
\rkf(\M):=\left( \brk^i_{\M, \s} \right)_{0 \leq i \leq n-1}.
\end{equation}
\end{definition}

\begin{rmk} 
\begin{enumerate}[(a)]

\item The choice of our terminology for $\rkf(\M)$ is justified by $(\ref{filt-Jordan-modules})$ which yields
$$
\mathbf{0}=\brk^n_{\M, \s}\leq \brk^{n-1}_{\M, \s} \leq \ldots \leq \brk^1_{\M, \s} \leq \brk^0_{\M, \s}=\brk_{\M_\s}.
$$

\bigskip
\item The degree zero rank invariant can be easily expressed in terms of the ranks of the structural maps of the module $\M$. Specifically, we have that
$$
\begin{aligned}
\brk^0_{\M, \s}:& \pos\times \pos \to \NN\cup \{\infty\}\\
& (x,y) \to
\begin{cases}
\rk \left( \restr{\M[x]}{\s} \to \restr{\M[y]}{\s}  \right) & \text{~if~} x \leq y\\
\infty & \text{~otherwise~}
\end{cases}
\end{aligned}
$$
\noindent
and thus
\begin{equation}\label{rk-zero-formula}
\brk^0_{\M, \s}(x,y)=\sum_{z \in \s} \rk \left( M_{y+z,x+z} \right),
\end{equation}
for all $x\leq y$.

\bigskip
\item When $i=1$, $\brk^1_{\M, \s}(x,x)$ is the rank of the operator $\T_{\M[x], \s}$ which, as a block matrix, has all block entries zero except the entries on the subdiagonal; see $(\ref{nil-op-as-block-matrix})$. Thus we can write
$$
\brk^1_{\M, \s}(x,x)=\sum_{l=1}^{n-1} \rk\left( T_{\M[x], \s}(l) \right).
$$

\bigskip
\item For any $i \geq 0$, we have that
$$
\brk^i_{\M,\s}(x, x)= \rk(\T^i_{\M[x], \s}), \forall x \in \pos.
$$ 
In particular, $\brk^i_{\M,\s}(\mathbf{0}, \mathbf{0})= \rk(\T^i_{\M, \s})$.
\end{enumerate}
\end{rmk}

Before we move on, let us explain the relationship between $\rkf$ and the Jordan type $\jt_\s$.

\begin{lemma}\label{relation-Jt-Filtered-rk-inv}
For any $x \in \pos$, the Jordan type of the shift $\M[x]$ can be expressed in terms of the filtered rank invariant of $\M$ at $\s$ as follows
\begin{equation}
\jt_\s(\M[x])=\left( \brk^{i+1}_{\M, \s}(x,x)+\brk^{i-1}_{\M, \s}(x,x)-2\brk^i_{\M, \s}(x,x)\right)_{i\in [n]}.
\end{equation}
\end{lemma}

\begin{proof} The formula follows at once from the general fact that expresses the number of Jordan blocks of size $i\times i$ of a nilpotent operator $T$ as $\rk(T^{i+1})+\rk(T^{i-1})-2\rk(T^i)$.
\end{proof}

\subsection{$\rkf$ is strictly finer than the classical rank invariant}
Let $\Gscr=[0,\ell_1]\times \ldots \times [0,\ell_d]\subseteq \pos=\ZZ^d$ be a $d$-dimensional grid. Let us slice it according to the norms of its points:
$$
S_{i}:=\{x \in \Gscr \mid |x|=i-1\}, \forall i \in [n],
$$
where $n:=\ell_1+\ldots+\ell_d+1$, and set
$$
\s=(S_1, \ldots, S_n).
$$

\begin{prop}\label{J-rk-finer-rk}
Keep the same notation as above. Then the Jordan filtered rank invariant $\rkf$ is strictly finer than the classical rank invariant.
\begin{enumerate}
\item Let $M$ and $N$ be two $\Gscr$-modules (extended trivially to the whole $\ZZ^d$). Then
$$
\brk^0_{M, \s}=\brk^0_{N, \s} \Longleftrightarrow \brk_M=\brk_N.
$$

\item Assume $d \geq 2$. There are (non-isomorphic) $\Gscr$-modules $X$ and $Y$ such that
$$
\brk^1_{X,\s}(\mathbf{0}, \mathbf{0}) \neq \brk^1_{Y, \s}(\mathbf{0}, \mathbf{0})\text{~but~} \brk_X=\brk_Y.
$$
\end{enumerate}
\end{prop}

\begin{proof} $(1)$ The implication $``\Longleftarrow"$ follows at once from $(\ref{rk-zero-formula})$. 

Now, let us assume that $\brk^0_{M, \s}=\brk^0_{N, \s}$, and let $\mathbf{x}, \mathbf{y} \in \Gscr$ be such that $\mathbf{x} \leq \mathbf{y}$. Note that for any $\Gscr$-module $L$, $(\ref{rk-zero-formula})$ yields
\begin{equation}\label{formula-proof-finer-rk-prop}
\begin{aligned}
\brk^0_{L, \s}(\mathbf{x}-(\ell_1, \ldots, \ell_d), & \mathbf{y}-(\ell_1, \ldots, \ell_d)) = \\ & \brk_L(\mathbf{x}, \mathbf{y})+\sum_{\mathbf{z}} \brk_L( \mathbf{x}-(\ell_1, \ldots, \ell_d)+\mathbf{z}, \mathbf{y}-(\ell_1, \ldots, \ell_d)+\mathbf{z}),
\end{aligned}
\end{equation}
where the sum is over all elements $\mathbf{z} \in \Gscr$ such that $\mathbf{z} \neq (\ell_1, \ldots, \ell_d)$ and $\ell_i-x_i \leq z_i \leq 2\ell_i-y_i$ for all $i \in [d]$.

Using induction of the norm of $\mathbf{y} \in \Gscr$ and $(\ref{formula-proof-finer-rk-prop})$, we get that $\brk_M(\mathbf{x}, \mathbf{y})=\brk_N(\mathbf{x}, \mathbf{y})$ for all $\mathbf{x} \leq \mathbf{y}$.

\bigskip
\noindent
$(2)$ Consider the zigzag poset 
$
\vcenter{\hbox{  
\begin{tikzpicture}[point/.style={shape=circle, fill=black, scale=.3pt,outer sep=3pt},>=latex]
   \node[point] (1) at (0,0) {};
   \node[point] (2) at (1,0) {};
   \node[point] (3) at (2,0) {};
   
   \path[->]
  (2) edge  (1)
  (2) edge (3);
\end{tikzpicture} 
}}
$
which we also view as an interval of $\Gscr$ by identifying the sink vertices with $(1,0, \ldots, 0)$ and $(0,1,0, \ldots, 0)$ and the source vertex with the origin. Let $X$ and $Y$ be the zigzag modules defined as
$$
X:~\vcenter{\hbox{  
\begin{tikzpicture}[point/.style={shape=circle, fill=black, scale=.3pt,outer sep=3pt},>=latex]
   \node[point, label={below:$K$}] (1) at (0,0) {};
   \node[point, label={below:$K^2$}] (2) at (1.25,0) {};
   \node[point, label={below:$K$}] (3) at (2.5,0) {};
   
   \path[->]
  (2) edge node[above] {$[1~0]$}  (1)
  (2) edge node[above] {$[1~0]$}  (3);
\end{tikzpicture} 
}}
\text{~and~}
Y:~\vcenter{\hbox{  
\begin{tikzpicture}[point/.style={shape=circle, fill=black, scale=.3pt,outer sep=3pt},>=latex]
   \node[point, label={below:$K$}] (1) at (0,0) {};
   \node[point, label={below:$K^2$}] (2) at (1.5,0) {};
   \node[point, label={below:$K$}] (3) at (3,0) {};
   
   \path[->]
  (2) edge node[above] {$[1~0]$}  (1)
  (2) edge node[above] {$[0~1]$}  (3);
\end{tikzpicture} 
}}
$$
It is then immediate to check that $\brk_X=\brk_Y$ and
$$
\brk^1_{X,\s}(\mathbf{0}, \mathbf{0})=1 \neq 2=\brk^1_{Y, \s}(\mathbf{0}, \mathbf{0}).
$$
\end{proof}

\section{Erosion-type and landscape stability for filtered rank invariants}
Throughout this section, we assume that $\pos=\ZZ^d$ or $\RR^d$. For a $\pos$-module $\M$ and $\varepsilon \in \pos$, the $\varepsilon$-shift of $\M$, denoted by $\M[\varepsilon]$, is the $\pos$-module defined by
$$
\left( \M[\varepsilon] \right)_x:=\M_{x+\varepsilon}, \forall x \in \pos, \text{~and~} \left( \M[\varepsilon] \right)_{yx}:=\M_{y+\varepsilon, x+\varepsilon}, \forall x \leq y.
$$
For $\varepsilon \in \pos$, we can define the $\varepsilon$-shift homomorphism 
$$
\sh^{\varepsilon}_{\M}:=\left(\M_{x+\varepsilon, x}:\M_x \to \M_{x+\varepsilon} \right)_{x \in \pos}:\M \to \M[\varepsilon].
$$ 
\noindent
For two modules $M$ and $N$,  an $\varepsilon$-interleaving between them is a pair of homomorphisms 
$$
\varphi: M \to N[\varepsilon] \qquad \text{and} \qquad \psi: N \to M[\varepsilon]
$$
such that 
$$
\psi[\varepsilon] \circ \varphi = \sh_M^{2\varepsilon}
\qquad \text{and} \qquad
\varphi[\varepsilon] \circ \psi =\sh_N^{2\varepsilon}.
$$

The interleaving distance between $M$ and $N$ is
$$
d_I(M,N) = \inf \{ \varepsilon \ge 0 \mid \text{there exists an $\overrightarrow{\varepsilon}$-interleaving between $M$ and $N$} \}.
$$
Here, $\overrightarrow{\varepsilon}$ denotes $(\varepsilon, \ldots, \varepsilon)$ for any $\varepsilon \in \RR_{\geq 0}$ or $\ZZ_{\geq 0}$.

\begin{rmk}
In \cite{Mike-Lesnick-2015}, Lesnick showed that the interleaving distance is the most discriminating and stable pseudo-metric for multiparameter persistence modules. It is arguably the best at capturing the notion of algebraic nearness.  

However, despite these theoretical advantages, computing the interleaving distance in practice is highly impractical, as it is known that for a given $\delta>0$, deciding whether $d_I \leq \delta$ is NP-hard when $d \geq 2$ (see \cite{interleaving-np-2020}). 
\end{rmk}

\subsection{The erosion distance and persistence landscapes} Let $F, G: \pos^{opp}\times \pos \to \left( \NN\cup \{\infty\} \right)^{opp}$ be two functors, i.e., $F$ and $G$ are functions such that
\begin{equation}\label{eqn-erosion-fun}
F(x,y) \leq F(x', y') \text{~and~} G(x,y) \leq G(x', y'),
\end{equation}
for all $x \leq x' \leq y' \leq y$. 

\begin{example} For a $\pos$-module $\M$, its rank invariant
$$
\begin{aligned}
\brk_\M:& \pos\times \pos \to \NN\cup \{\infty\}\\
& (x,y) \to
\begin{cases}
\rk\left( \M_{yx} \right)  & \text{~if~} x \leq y\\
\infty \hspace{50pt} & \text{~otherwise~}
\end{cases}
\end{aligned}
$$
can be easily seen to satisfy $(\ref{eqn-erosion-fun})$.
\end{example}

We now come to a notion of distance introduced in the context of MPH by Patel \cite{Patel-erosion-2018}, Puuska \cite{Puuska-erosion-2020}, and Kim-M{\'e}moli \cite{Kim-Memoli-erosion-2021}. 

\begin{definition} The \emph{erosion distance} between $F$ and $G$ is
\begin{equation}
d_E(F,G):=\inf \left\lbrace \varepsilon \geq 0 \;\middle|\;
\begin{array}{l}
F(x-\overrightarrow{\varepsilon}, y+\overrightarrow{\varepsilon}) \leq G(x,y),\text{~and}\\
G(x-\overrightarrow{\varepsilon}, y+\overrightarrow{\varepsilon}) \leq F(x,y), \forall x,y \in \pos
\end{array}
\right\rbrace.
\end{equation}
\end{definition}
\begin{rmk}
In the $1$-parameter case, the erosion distance is a lower bound for the bottleneck distance. Furthermore,  it is believed it can be computed effectively (see \cite[Example 1]{Adams-et-al-2022}). For a general discussion on the computational complexity of $d_E$, see \cite{Kim-Memoli-erosion-2021}.  
\end{rmk}

In \cite{Peter-landscapes-2015}, Bubenik introduced and studied the persistence landscape of a persistence module, providing stable vectorizations of persistence diagrams.  Later on, Vipond \cite{Vipond-persistence-multi-2020} extended this notion to multiparameter persistence modules.

\begin{definition}
Given a functor $F: \pos^{opp}\times \pos \to \left( \NN\cup \{\infty\} \right)^{opp}$, its \emph{persistence landscape} is the function $\lambda_F \in L^{\infty}(\NN \times \pos) $ defined by
\begin{equation}
\lambda_F(k,x):=\sup \{  \varepsilon >0 \mid F(x-\mathbf{h}, x+\mathbf{h})\geq k, \forall \norm{\mathbf{h}}_{\infty} \leq \varepsilon \}.
\end{equation}
\noindent
If $F$ is of the form $\brk_\M$ for a $\pos$-module $\M$, we write $\lambda_\M$ for $\lambda_{\brk_{\M}}$.
\end{definition}

\begin{theorem}\label{thm-general-landscape-erosion}
\begin{enumerate}
\item Let $F,G: \pos^{opp}\times \pos \to \left( \NN\cup \{\infty\} \right)^{opp}$ be two functors with associated landscapes $\lambda_F, \lambda_G \in L^{\infty}(\NN \times \pos)$. Then
$$
\norm{ \lambda_F-\lambda_G }_{\infty} \leq d_E(F,G).
$$

\item Assume $F=\brk_M$ and $G=\brk_N$ where $M$ and $N$ are two $\pos$-modules. Then
$$
d_E(\brk_M,\brk_N) \leq d_I(M,N),
$$
where $d_I$ is the interleaving distance. Consequently, 
$$
\norm{ \lambda_M-\lambda_N }_{\infty} \leq d_E(\brk_M,\brk_N) \leq d_I(M,N).
$$
\end{enumerate}
\end{theorem}
The first part of Theorem \ref{thm-general-landscape-erosion} was proved in \cite[Proposition 3.11]{Marc-Fer-2024} and the second part in \cite[Theorem 30]{Vipond-persistence-multi-2020}.

\subsection{A template for achieving erosion-type stability} In this section, we describe a template, inspired by \cite{Marc-Fer-2024}, for establishing stability results for invariants constructed within a specific framework. Let $\A$ be a class of $\pos$-modules of interest, closed to isomorphisms and shifts, and let $\F$ be a functor from $\A$ to $\vect_K$, the category of finite-dimensional $K$-vector spaces. 

With the functor $\F$ in place, given a module $\M \in \A$, we consider the $\pos$-module $\M^\F$ defined by
\begin{equation}
\M^\F_x:=\F(\M[x]), \text{~for all~} x \in \pos, \text{~and~} \M^\F_{y,x}:=\F\left( \sh^{y-x}_{\M[x]}\right), \text{~for all~} x \leq y.
\end{equation}

\begin{definition}[\textbf{The $\F$-rank invariant}] We define the \emph{$\F$-rank invariant of $\M$}, denoted by $\brk_{\M,\F}$, to be the rank invariant of the $\pos$-module $\M^\F$, \emph{i.e.},

\begin{equation}
\begin{aligned}
\brk_{\M, \F}:& \pos\times \pos \to \ZZ_{\geq 0}\cup \{\infty\}\\
& (x,y) \to
\begin{cases}
\rk\left( \M^\F_{y,x} \right)  & \text{~if~} x \leq y\\
\infty & \text{~otherwise.~}
\end{cases}
\end{aligned}
\end{equation}
\end{definition}
\noindent
As a rank invariant, $\brk_{\M, \F}$ can also be viewed as a functor from $\pos^{opp} \times \pos$ to $\left( \ZZ_{\geq 0}\cup \{\infty\} \right)^{opp}$.

The following lemma, although straightforward, plays a crucial role in proving erosion-type stability results for $\F$-rank invariants. 

\begin{lemma}(compare to \cite[Lemma 3.6]{Marc-Fer-2024}) \label{main-lemma-stability} Let $M$ and $N$ be $\pos$-modules from $\A$, and let $x\leq x'\leq y'\leq y$ be elements of $\pos$. Assume that there are homomorphims 
$$
f:M \to N[x'-x] \text{~and~} g:N \to M[y-y']
$$
such that
\begin{equation}\label{translates-equation-key-lemma}
g[x'-x] \circ f =\sh_M^{(y-x)-(y'-x')}.
\end{equation}
Then
\begin{equation}
\brk_{M, \F}(x,y) \leq \brk_{N, \F}(x',y').
\end{equation}
\end{lemma}

\begin{rmk}
The motivation behind $(\ref{translates-equation-key-lemma})$ arises from the notion of interleaving homomorphisms.  Indeed, if $x, x', y$, and $y'$ satisfy $x'-x=y-y'=\overrightarrow{\varepsilon}$ with $\varepsilon \geq 0$, then $(\ref{translates-equation-key-lemma})$ becomes 
$$
g[\overrightarrow{\varepsilon}]\circ f=\sh_M^{2 \overrightarrow{\varepsilon}},
$$
which is one the equations that occurs in the definition of a pair of interleaving homomorphisms.
\end{rmk}

\begin{proof} We begin by claiming that
\begin{equation} \label{2nd-eq-key-lemma-stability}
g[y']\circ \sh_{N[x']}^{y'-x'} \circ f[x]=\sh_{M[x]}^{y-x}.
\end{equation}
Indeed, for any $z \in \pos$, the $z$-component of the composition of the maps on the left-hand side of $(\ref{2nd-eq-key-lemma-stability})$ is
$$
g_{y'+z}\circ N_{y'+z,x'+z}\circ f_{x+z} \in \Hom_K(M_{x+z}, M_{y+z}).
$$
Next, re-writing $(\ref{translates-equation-key-lemma})$ at the level of linear maps yields
$$
g_{x'-x+u}\circ f_u=M_{y-y'+x'-x+u, u}, \forall u \in \pos.
$$
Letting $u=x+z$ in the equation above and using the fact that $g$ is a homomorphism, we obtain the commutative diagram
\[
\begin{tikzpicture}[>=latex, node distance=3cm]
    \node (V1) {$M_{x+z}$};
    \node (V2) [right of=V1] {$N_{x'+z}$};
    \node (V3) [right of=V2] {$M_{y-y'+x'+z}$};
    \node (W1) [below of=V2] {$N_{y'+z}$};
    \node (W2) [right of=W1] {$M_{y+z}$};

    \draw[->] (V1) -- (V2) node[midway, above] {$f_{x+z}$};
    \draw[->] (V2) -- (V3) node[midway, above] {$g_{x'+z}$};
    \draw[->] (V2) -- (W1) node[midway, left] {$N_{y'+z, x'+z}$};
    \draw[->] (V3) -- (W2) node[midway, right] {$M_{y+z, y-y'+z}$};
    \draw[->] (W1) -- (W2) node[midway, below] {$g_{y'+z}$};

    \draw[->, bend left=28] (V1) to node[midway, above] {$M_{y-y'+x'+z, x+z}$} (V3);
    
    \node at ($(V2)!0.5!(W2)$) {\(\circlearrowleft\)};
\end{tikzpicture}
\]
\noindent
Composing the linear maps along the two paths from $M_{x+z}$ to $M_{y+z}$, we obtain
$$
g_{y'+z}\circ N_{y'+z,x'+z}\circ f_{x+z} =M_{y+z, y-y'+z} \circ M_{y-y'+x'+z, x+z}=M_{y+z, x+z},
$$
and this proves our claim.

Finally, applying the functor $\F$ to $(\ref{2nd-eq-key-lemma-stability})$, we obtain the commutative diagram 
\[
\begin{tikzpicture}[node distance=2cm, auto]
    \node (1) at (0, 0) {$\F(M[x])$};
    \node (2) at (4.5, 0) {$\F(M[y])$};
    \node (3) at (0, -2) {$\F(N[x'])$};
    \node (4) at (4.5, -2) {$\F(N[y'])$};
    
    \draw[->] (1) -- (2) node[midway, above] {$\F(\sh^{y-x}_{M[x]})$};
    \draw[->] (1) -- (3) node[midway, left] {$\F(f[x])$};
    \draw[->] (4) -- (2) node[midway, right] {$\F(g[y'])$};
    \draw[->] (3) -- (4) node[midway, below] {$\F(\sh^{y'-x'}_{N[x']})$};

    \node at ($(1)!0.5!(4)$) {\(\circlearrowleft\)};
\end{tikzpicture}
\]
This implies that $\rk(\F(\sh^{y-x}_{M[x]})) \leq \rk(\F(\sh^{y'-x'}_{N[x']}))$, which is what we needed to prove.
\end{proof}
We are now ready to prove the following result, our template for achieving stability.

\begin{theorem}(compare to \cite[Theorem 3.7]{Marc-Fer-2024})\label{thm-stability-template} For two modules $M,N \in \A$, 
$$
d_E(\brk_{M^\F}, \brk_{N^\F}) \leq d_I(M,N).
$$
\end{theorem}

\begin{proof} Let $\varepsilon \geq 0$ be such that there exist interleaving homomorphisms $f:M \to N[\overrightarrow{\varepsilon}] \text{~and~} g:N \to M[\overrightarrow{\varepsilon}]$ satisfying
\begin{equation}\label{eqn-thm-template}
g[\overrightarrow{\varepsilon}]\circ f=\sh_M^{2 \overrightarrow{\varepsilon}}
\text{~and~}
f[\overrightarrow{\varepsilon}]\circ g=\sh_N^{2 \overrightarrow{\varepsilon}}.
\end{equation}

\bigskip
Now let $x, y \in \pos$ be such that $x \leq y$. Since $x-\overrightarrow{\varepsilon} \leq x\leq y\leq y+\overrightarrow{\varepsilon}$, we get from $(\ref{eqn-thm-template})$ and Lemma \ref{main-lemma-stability} that
$$
\brk_{M, \F}(x-\overrightarrow{\varepsilon}, y+\overrightarrow{\varepsilon}) \leq \brk_{N, \F}(x,y) \text{~and~} \brk_{N, \F}(x-\overrightarrow{\varepsilon}, y+\overrightarrow{\varepsilon}) \leq \brk_{M, \F}(x,y).
$$
The proof now follows.
\end{proof}

Finally, we are in a position to prove our stability theorem.

\begin{proof}[\textbf{Proof of Theorem \ref{main-thm-Jt-stability}}]
For each $i \in \{0,\ldots,  n-1 \}$, applying Theorem \ref{thm-stability-template} to the functor 
$$L \to \Ima \left( \T^i_{L,\s}\right),
$$
we obtain
\begin{equation}\label{ineq-thm-2-proof}
d_E(\brk^i_{\M, \s}, \brk^i_{N,\s}) \leq d_I(M,N),
\end{equation}
and this yields the first inequality of the theorem.

Finally, as a consequence of Theorem \ref{thm-general-landscape-erosion} and $(\ref{ineq-thm-2-proof})$, we get that
$$
\max_{i \in \{0, \ldots, n-1
\}} \norm{ \lambda_{M^i_\s}-\lambda_{N^i_\s} }_{\infty} \leq d_{E}(M,N)_\s.
$$
\end{proof}

\subsection*{Acknowledgment} 
The first author would like to thank Thomas Br{\"u}stle, Peter Bubenik, Justin Desrochers, Emerson Escolar, Erik Hanson, and Woojin Kim for stimulating discussions related to parts of this work. C. Chindris is supported by Simons Foundation grant $\# 711639$.


\end{document}